\documentclass[11pt]{amsart}
\usepackage[usenames,dvipsnames,svgnames,table]{xcolor}
\usepackage[pagebackref,linktocpage=true,colorlinks=true,linkcolor=Blue,citecolor=BrickRed,urlcolor=RoyalBlue]{hyperref}
\usepackage[alphabetic]{amsrefs}
\usepackage{skak} 

\usepackage{amsmath,amssymb,amsfonts,amsthm,enumerate,textcomp}
\usepackage{url}
\usepackage{graphicx}
\usepackage{xcolor}
\usepackage{wrapfig}

\usepackage[mathscr]{euscript} 

\usepackage{esint}
\usepackage{enumitem}

\usepackage[left=1.1in,right=1.1in,top=1.3in,bottom=1.3in]{geometry}

\usepackage[colorinlistoftodos,prependcaption,textsize=tiny]{todonotes}
\usepackage{amssymb,amscd,amsthm,amsxtra}

\newtheorem{thm}{Theorem}[section]
\newtheorem{cor}[thm]{Corollary}
\newtheorem{lem}[thm]{Lemma}
\newtheorem{prop}[thm]{Proposition}
\theoremstyle{definition}
\newtheorem{defn}[thm]{Definition}
\newtheorem{claim}[thm]{Claim}

\newtheorem{rem}[thm]{Remark}

\newtheoremstyle{named}{}{}{\itshape}{}{\bfseries}{.}{.5em}{\thmnote{#3 }#1}
\theoremstyle{named}

\numberwithin{equation}{section}
\newtheorem*{AAA}{Theorem A}

\newtheorem*{BBB}{Theorem B}
\newtheorem*{CCC}{Theorem C}

\newcommand{\R}{\mathbb R}

\newcommand{\p}{\partial}

\newcommand\F{\mathcal F}

\renewcommand{\H}{\mathcal{H}}

\renewcommand{\l}{\langle}
\renewcommand{\r}{\rangle}
\renewcommand{\(}{\left(}
\renewcommand{\)}{\right)}
\renewcommand{\tilde}{\widetilde}

\newcommand\Om{\Omega}
\newcommand\om{\omega}
\renewcommand\o[1]{\omega({#1})}

\def\na{\nabla}
\newcommand\e{\varepsilon}
\def\supp{{\rm{supp}}}

\def\diam{\mbox{diam}}
\newcommand{\I}[1]{\chi_{\{#1>0\}}}
\newcommand\M[1]{\mathcal M(\R^{#1})}
\renewcommand\div{\operatorname{div}}
\renewcommand\phi{\varphi}

\usepackage{tikz}
\usetikzlibrary{arrows,shapes,positioning}
\usetikzlibrary{decorations.markings}
\tikzstyle arrowstyle=[scale=1]
\tikzstyle directed=[postaction={decorate,decoration={markings,
    mark=at position .65 with {\arrow[arrowstyle]{stealth}}}}]
\tikzstyle reverse directed=[postaction={decorate,decoration={markings,
    mark=at position .65 with {\arrowreversed[arrowstyle]{stealth};}}}]
\usetikzlibrary{patterns}

\begin{document}

\title[A nonlocal free boundary problem with Wasserstein distance]
{A nonlocal free boundary problem with Wasserstein distance}

\author{Aram L. Karakhanyan}
\address{}
\email{aram6k@gmail.com}

\address{ School of Mathematics, The University of Edinburgh, Peter Tait Guthrie Road, EH9 3FD
Edinburgh, UK}

\thanks{ }

\begin{abstract}
We study the probability measures $\rho\in \M2$ minimizing the functional 
\[
J[\rho]=\iint \log\frac1{|x-y|}d\rho(x)d\rho(y)+d^2(\rho, \rho_0),
\]
where $\rho_0$ is a given probability measure and $d(\rho, \rho_0)$ is the 2-Wasserstein distance 
of $\rho$ and $\rho_0$.  

We prove the existence of minimizers $\rho$ and show that the potential $U^\rho=-\log|x|\ast \rho$ 
solves a degenerate obstacle problem, the obstacle being the transport potential.
Every minimizer $\rho$ is absolutely continuous with respect to the Lebesgue measure.
The singular set of the free boundary of the obstacle problem is contained in a  rectifiable set, and its Hausdorff  dimension is $< n-1$.
Moreover, $U^\rho$ solves a \textit{nonlocal} Monge-Amp\`ere equation, which after linearization leads to the 
equation $\rho_t=\div(\rho\na U^\rho)$. 
The methods we develop use \textit{Fourier transform} techniques. They work equally well in high dimensions $n\ge2$ 
for the energy 
\[
J[\rho]=\iint |x-y|^{2-n}d\rho(x)d\rho(y)+d^2(\rho, \rho_0).
\]

\end{abstract}
\maketitle

\section{Introduction}\label{sec1}

In this paper we are concerned with  the minimization of the functional
\begin{equation}\label{eq:1.1}
J[\rho]=\iint \log\frac1{|x-y|}d\rho(x)d\rho(y)+d^2(\rho, \rho_0)
\end{equation}
among all probability measures $\rho$ with finite second momentum. 
Here $d^2(\rho, \rho_0)=\inf_{\gamma}\frac12 \iint |x-y|^2d\gamma(x,y)$ is the 
square of the Wasserstein distance  between $\rho$ and a
given probability measure $\rho_0$, and $\gamma$ is a joint probability  measure 
with marginals 
$\pi_{x\#}\gamma=\rho$, $\pi_{y\#}\gamma=\rho_0$. 
The support of  $\rho$ is a priori unknown (or free) and our main goal is to analyze the regularity of the free boundary, i.e. the boundary of the set where $\rho\not =0$. 

An analogous problem arises in  high dimensions if we replace the logarithmic 
kernel by $K(x-y)={|x-y|^{2-n}}, n\ge 3$. The methods we employ do not depend on the dimension. We focus on the logarithmic kernels since the potential $U^\rho=-\rho\ast \log|x|$ may change sign and log-interaction phenomenon
has a number of important applications \cite{ST}, \cite{Serfaty} (in Section 2 we also give a connection with random matrices).

An interesting feature of the variational problem for $J[\rho]$ is that 
it leads to an obstacle problem involving the potential of the optimal transport of $\rho$ to $\rho_0$.
Let $U^\rho$ be the logarithmic (or the Newtonian potential if $n\ge 3$) of the probability measure 
$\rho$ and $\psi$ the potential of the transport map, then formally we have 
\begin{equation}\label{eq:f1}
U^\rho=\psi \quad \{\rho>0\}\ \mbox{and}\  U^\rho\ge \psi\ \mbox{elsewhere}.
\end{equation}
Since $\Delta U^\rho=-2\pi \rho$ then it follows that  
\begin{equation}\label{eq:f2}
\Delta U^\rho=\Delta \psi\quad\mbox{in}\ \{\rho>0\},  \quad \Delta U^\rho=0\quad\mbox{in}\ \{\rho=0\}.
\end{equation}
Thus combining \eqref{eq:f1} and \eqref{eq:f2} we have the obstacle problem  
\begin{equation}\label{eq:main-obs}
\left\{
\begin{array}{ccc}
\Delta U^\rho= \Delta\psi \I \rho &\quad \mbox{in}\ \R^2,  \\
\rho(U^\rho-\psi)=0  &\quad \mbox{in}\ \R^2. 
\end{array}
\right.
\end{equation}
In this formulation the position of the obstacle is a priori unknown as opposed to the classical case \cite{C-ob}. 
Note that $\psi$ is semiconvex function, hence 
from Aleksandrov's theorem it follows that $D^2 \psi$ exists a.e. Consequently,  the 
first equation in \eqref{eq:main-obs} is satisfied in a.e. sense provided that $\rho$ is absolutely 
continuous with respect to the Lebesgue measure.

The partial mass transport and Monge-Amp\`ere obstacle problems 
had been developed in the seminal work  of  Caffarelli and McCann \cite{CM}, see also  \cite{Figalli},
\cite{DF} and the references given there.

Several papers introduced variational problems for measures. 
In \cite{MC} McCann formulated a variational  principle for the energy 
\[
E[\rho]=\int A(\rho)+\frac12\iint d\rho(x)K(x-y)d\rho(y), 
\]
which allowed to prove 
existence and uniqueness for  a family of attracting gas models, and generalized the 
Brunn-Minkowski inequality from sets to measures.

Another interesting energy 
\[
F[\rho]=\iint\log\frac1{|x-y|}d\rho(x)d\rho(y)+\int|x|^2d\rho, 
\] 
appears in the large deviation laws and log-gas interactions \cite{Serfaty}, \cite{ST}. Thanks to the  quadratic potential 
every measure minimizing $F[\cdot]$ is confined in some ball. 
Furthermore, one can prove  transport inequalities and bounds for the 
Wasserstein distance in terms of $F[\rho]$  \cite{Ledoux}. 

There is a vast literature on interaction energies for probability measures governed by the Wasserstein metric \cite{JKO}, 
\cite{CMV}, 
\cite{C-D},  
\cite{C-ARMA}.
In particular, \cite{MR3805040}  contains an $L^\infty$ estimate for the equilibrium measure 
and 
\cite{MR3488544} a connection to obstacle problems.

In \cite{S} Savin considered the optimal  transport of the probability measures in periodic setting for the 
energy $
\int|\na \rho|^2+d^2(\rho, \rho_0), \rho\in H^1([0,1]^n)
$. 
The resulted obstacle problem takes the form 
\begin{equation}\label{eq:obst-Sav}
\left\{
\begin{array}{ll}
-\Delta \rho=\psi\ &\mbox{in}\ \{\rho>0\}, \ \mbox{} \\
-\Delta \rho\ge \psi \ &\mbox{elsewhare},
\end{array}
\right.
\end{equation}
where $\psi$ is the transport potential of $\rho\to \rho_0$ with given initial periodic probability measure $\rho_0$ with $H^1$ density.

The aim of this paper is to bring together two areas  in which the nonlocal interactions are confined by the square of Wasserstein's  
distance.

\smallskip

\subsection*{Main results}
The energy $J[\rho]$ has nonlocal character due to the presence of the logarithmic 
kernel. However, thanks to the Wasserstein distance $\rho$ is forced to have   compact support 
provided that $\supp \rho_0$ is compact.  Observe that 
if $\rho$ has atoms then $J[\rho]=\infty$. 
\begin{AAA}
If $\rho_0$ has compact support then there is a probability measure $\rho$ minimizing 
$J$ such that $\supp\rho$ is compact.
Moreover, $\rho$ cannot have atoms
and hence there is a measure preserving transport map $y=T(x)$ such that  $\rho_0$ is the push forward of $\rho$.  
\end{AAA}

The second part of the theorem follows from the standard theory of optimal transport \cite{Ambrosio}.
The chief difficulty in proving the first part is to show that there is a minimizing sequence of 
probability measure with uniformly bounded supports. In order to establish this we use 
Carleson's estimate from below for the nonlocal term and  a localization argument  for the 
Fourier transforms of these measures.

Next we want to analyze the character  of equilibrium measures. Since the 
problem involves mass transport then there must be some hidden convexity related to $\rho$.
To see this we compute and explore the first variation of $J$. The weak form of the 
Euler-Lagrange  equation implies that $\hat \rho$, the Fourier transform of $\rho$,  is in 
$L^2$.

\begin{BBB}
Let $\rho$ be a minimizer. Then  $\widehat \rho\in L^2(\R^2)$
and $d\rho=f dx$ on $\supp \rho$ where $f\in L^{\infty}(\R^2)$. 
In particular, the transport map $y=T(x)$ (as in Theorem A) is given by 
\[
y=x+2\na U^\rho, 
\]
where $U^\rho=\rho\ast K$ is the potential of $\rho$ and $\na U^\rho$
is log-Lipschitz continuous. 
\end{BBB}

The log-Lipschitz continuity of $\na U^\rho$ follows from 
Judovi\v{c}'s therem \cite{Yudovich}. In fact from the Calder\`on-Zygmund estimates 
it follows that 
$D^2U^\rho\in L^p_{loc}$ for every $p>1$. The local mass balance condition for the optimal transport 
leads to a nonlocal Monge-Amp\`ere equation 
\begin{equation}\label{eq:nonloc-MA}
\det(Id+2D^2 U^\rho)=\frac{\rho(x)}{\rho_0(x+2\na U^\rho)}. 
\end{equation}
\eqref{eq:nonloc-MA} implies that $\supp \rho\subset \supp \rho_0$. 
If we linearize \eqref{eq:nonloc-MA} using a time discretization scheme,  the resulted equation is
$\rho_t=\div(\rho\nabla U^\rho)$. 

The analysis of the structure of singular set in the obstacle problems is the central problem of the regularity theory. 
Let $\hbox{MD}(\supp\rho\cap B_r(x))$ be the infimum of distances between pairs of 
parallel planes such that $\supp\rho\cap B_r(x)$ is contained in the strip determined by them \cite{C-ob}. Let 
\begin{equation}\label{small-o}
\o R=\sup_{r\le R}\sup_{x\in \supp \rho} \frac{\hbox{MD}(\supp\rho\cap B_r(x))}r.
\end{equation}
Observe that if $n=2$ then \eqref{eq:nonloc-MA} is equivalent to 
$2\pi\rho_0[4\det D^2U^\rho+2 \Delta U^\rho+1]=-\Delta U^\rho$. From here we can deduce  the equation
\begin{equation}\label{bobb}
\det\left[2D^2U^\rho+Id\left(1+\frac1{4\pi\rho_0}\right)\right]=\left(1+\frac1{4\pi\rho_0}\right)^2-1>0.
\end{equation}
Consequently, the standard regularity theory for the Monge-Amp\`ere equation (see \cite{TW}) implies that 
we can get higher regularity for $\rho$ if $\rho_0$ is sufficiently smooth.  
\begin{CCC}
Let $\o R$ be the modulus of continuity of the slab height (see \eqref{small-o}), 
$B_i=B_{r_i}(x_i)$ a collection of disjoint balls included in $B_R$ with $x_i\in S$,  where 
$S$ is the singular set. Then for every $\beta>n-1$ we have 
\[
\sum r_i^\beta\le C\frac{R^\beta}{\om^{n-1}(R)}\frac1{1-\om^{\beta-(n-1)}(R)}.
\]
Furthermore, if $\o R= R^\sigma$, then there is $\sigma'=\sigma'(n, \sigma)$such that the singular set $S\subset M_0\cup\bigcup_{i=1}^\infty M_i$
where $\H^{n-1-\sigma'}(M_0)=0$ and $M_i$ is contained in some $C^1$ hypersurface 
such that the measure theoretic normal exists at each $x\in S\cap M_i, i\ge 1$.

\end{CCC}

The paper is organized as follows: 
In Section \ref{sec:2} we recall some facts on the Wasserstein distance 
and Fourier transformation of measures. One of the key facts that we use is that 
the logarithmic term can be written as a weighted $L^2$ norm of the Fourier transformation of $\rho.$

Section \ref{sec:3} contains the proof of Theorem A. The chief difficulty in the proof is to control the supports of 
the sequence of minimizing measures.
In Section \ref{sec:4} we discuss the relation of $J[\rho]$ with the large deviations laws
for the random matrices with interaction and provide a simple model with energy $J$.

Section \ref{sec:6} contains some basic discussion of 
cyclic monotonicity and maximal Kantorovich potential. Then we derive 
the  Euler-Lagrange equation. From here we infer that $\rho$ has $L^\infty$
density with respect to the Lebesgue measure. Theorem B 
follows from Theorem \ref{thm:mon} and Corollary \ref{cor:pizza}.
Section \ref{sec:7} is devoted to the nonlocal Monge-Amp\`ere equation and its 
linearization $\rho_t=\div(\rho\na U^\rho)$.
Finally, in Section \ref{sec:8} we study the regularity of free boundary and prove Theorem C. 

\subsection*{Notation} 
We will denote by $\mathcal M(\R^n)$ the set of probability measures on $\R^n$, 
and let $\mu_{\# f}$ be the push forward of $\mu\in \M n$ under a mapping $f$. 
$d(\mu, \rho)$ denotes  the 2-Wasserstein distance of $\mu, \rho\in \M n$, $B_r(x_0)$  the open 
ball of radius $r$ centered at $x_0$, $K$ the kernels
\begin{equation}\label{eq:kernel-def}
K(x-y)=
\left\{
\begin{array}{ccc}
\log\frac1{|x-y|} \quad &\mbox{if}\ n=2,\\
|x-y|^{2-n} &\mbox{if}\ n\ge 3.
\end{array}
\right.
\end{equation}
$U^\rho=\rho\ast K$ is the potential of 
$\rho\in \M n$, $\H^n$  the $n$ dimensional Hausdorff measure, 
$1_E$  the characteristic function of $E\subset \R^n$. 
The restriction of $\mu \in \M n$ on some $E\subset \R^n$ will be denoted by $\mu\with E:=1_E\mu$,
and $\widehat \mu(\xi) =\int e^{-2\pi i \l x, \xi\r}d\mu(x)$ is the Fourier transform of $\mu\in \M n$.

\section{Set-up}\label{sec:2}

Let $f:\R^n\to\R^n$ be a map,  for a   Borel set $E\subset \R^n$ 
the push forward is defined by $\mu_{\# f}(E)=\mu(f^{-1}(E))$. 
For every joint probability measure $\gamma\in \mathcal M(\R^n\times \R^n)$ we define the projections
$\pi_x :(x, y)\to x$, $\pi_y :(x, y)\to y$.

We require $\gamma$ to have prescribed marginals $\rho$, $\rho_0\in \M n$, i.e.
\[
\gamma_{\#\pi_x}=\rho(x), \quad 
\gamma_{\#\pi_y}=\rho_0(y).
\]

For probability measures $\rho, \rho_0\in \mathcal M(\R^n)$ we define their Wasserstein distance as follows  
\begin{equation}
d(\rho, \rho_0)=\left(\inf_\gamma\frac12\iint |x-y|^2d\gamma(x, y)\right)^{\frac12}, 
\end{equation}
where $\gamma$'s are transport plans such that $\gamma_{\#\pi_x}=\rho$, 
$\gamma_{\#\pi_y}=\rho_0$.
We recall the following properties of the Wasserstein distance:
\medskip 
\begin{itemize}
\item[1)] $d$ is a distance, 
\item[2)] $d^2$ is convex, i.e.
\[
d^2(tu+(1-t)v, w)\le td^2(u, w)+(1-t)d^2(v, w), \quad t\in [0, 1], u, v\in \M n,
\]
\item[3)] if $u_k\to u, v_k\to v$ in $L^1_{loc}$ as $k\to \infty$ then 
\[
\lim_{k\to 0}d(u_k, v_k)=d(u, v),
\]
\item[4)] if $u_k\to u, v_k\to v$ weakly, i.e. $\int u_k\phi\to \int u\phi, \int v_k\phi\to \int v\phi$
for every $\phi\in C_0$, then 
\[
d(u, v)\le \liminf_{k\to \infty}d(u_k, v_k).
\]
\end{itemize}
See  \cite{Villani} for more details. 

We also need the following definition of Wasserstein class:
\begin{defn} 
Let $(\Omega, |\cdot |)$ be a Polish space (i.e. complete separable metric space equipped with its Borel $\sigma$-algebra). The Wasserstein space of order $2$ is defined as
\[
P_2(\Om)=\left\{\mu \in \mathcal M : \int_{\Om}|x_0-x|^2\mu(dx)<\infty\right\},
\]
where $x_0\in \Om$ is arbitrary. This space does not depend on the choice of $x_0$. 
Thus $d$ defines a finite distance on $P_2$. 
\end{defn}
\begin{rem}
If $\Om$ is compact then so is $P_2$. If $\Om$ is only locally compact then 
$P_2(\Om)$ is not locally compact, see \cite{Villani}.  This introduces several difficulties in the proof of the existence of a minimizer.  
\end{rem}

\begin{rem}\label{rem:Fourier}
Recall that the Fourier transformation of the truncated kernel $K_{r_0}=1_{B_{r_0}}K, n=2$ can be computed explicitly 
\begin{equation}\label{eq:K-Fourier}
\widehat{ K_{r_0}}=\frac{c_1}{4\pi|\xi|^2}(1-\mathcal B(2\pi r_0|\xi|)), 
\end{equation}
where $c_1>0$ is a universal constant, $\mathcal B$ is the Bessel function of the first kind such that $\mathcal B(0)=1, \mathcal B'(0)=0$
and $\lim_{t\to+ \infty}\mathcal B(t)=0$ \cite{Carleson}.

If $\mu\in \mathcal M(\R^2)$ has compact support then from the weak Parceval identity we have that 
\begin{equation}\label{grbr}
\iint K(x-y)\mu(x)\mu(y)=\int|\hat \mu |^2\widehat K\ge 0, 
\end{equation}
where $K(x-y)=\log\frac1{|x-y|}$ and $\widehat\mu, \widehat K$ are the Fourier transforms of $\mu, K$ respectively, see  \cite{K17} for the proof. This observation shows that the energy $J$ is nonnegative for compactly supported $\mu\in \M 2$.

We say that  $\mu \in \M n$ has finite  energy if $I[\mu]<\infty$ where $I[\rho]=\iint K(x-y)d\rho(x)d\rho(y)$. 
Then $\M n$ with $\mathcal I[\rho, \mu]=\iint K(x-y)d\rho(x)d\mu(y)$ has Hilbert structure, \cite{Landkof} page 82, and 

\[
\|\mu\|=\sqrt{\mathcal I[\mu, \mu]}
\]
is a norm.  It is  remarkable that the standard mollifications $\mu_k$ of $\mu$ converge 
to $\mu$ strongly, i.e. $\lim_{k \to \infty}\|\mu-\mu_k\|=0$, see \cite{Landkof} Lemma 1.$2'$ page 83.
\end{rem}

\section{Existence of minimizers}\label{sec:3}

\begin{prop}\label{main-prop}
Let $\mu_0\in \M 2$ and $\supp\mu_0\subset B_{R_0}$ for some $R_0>0$.
Let $\mu\in P_2(\R^2)$ and $J$ be given by \eqref{eq:1.1}, then 
\begin{itemize}
\item[(i)] $J[\mu]>-\infty$ provided that $J[\mu]<+\infty$, 
\item[(ii)] there is $\e>0$ depending on $R_0$ and $\mu$ such that 
$J[\mu_{\e}]<J[\mu]$ provuded that  $\supp\mu\not\subset {B_\e}$, 
where $\mu_{\e}=1_{B_\e}\mu/\mu(B_{\e})$ is the normalized restriction of $\mu$ to $B_{\e}$,
\item[(iii)] if $0\le J[\mu_k]\le C$ for some sequence $\{\mu_k\}\subset P_2(\R^2)$ 
and $\e_k$ are the corresponding numbers from (ii) then there is 
$\e_0>0$ such that $\e_k\le \e_0$ uniformly in $k$,  
where $\e_0$ depends only on $C$ and $R_0$.  
\end{itemize}

\end{prop}

\begin{proof}
We split the proof into three steps:

{\bf Step 1: Second momentum estimate:}

Let $\e>0$ be fixed. By Theorem 1 \cite{Rachev} there is transference plan $\gamma\in \mathcal M(\R^2\times B_{R_0}) $  with marginals $\mu, \mu_0$
such that $d^2(\mu, \mu_0)=\frac12\iint |x-y|^2\gamma.$  
Set $\gamma_\e(x,y)=\frac1{\mu(B_\e)}\gamma\with( B_{\e}\times\R^n)$
then 
\[
 \iint \gamma_\e
= 
\frac1{\mu(B_\e)}\iint\limits_{B_{\e}\times B_{R_0}} \gamma
=  
\frac1{\mu(B_\e)}\int_{B_{\e}} \mu(x)=1.
\]
Moreover, the projections of $\gamma_\e$ are $\mu_\e=\frac1{\mu(B_\e)}\mu\with B_\e$ and $\mu_0$. 
Hence 
\begin{eqnarray*}
d^2(\mu, \mu_0)
&=&
\frac12\iint|x-y|^2\gamma\\
&=& 
\frac 12\iint\limits_{B_\e\times B_{R_0}}|x-y|^2\gamma+ \frac12\iint\limits_{B_\e^c\times B_{R_0}}|x-y|^2\gamma\\
&=&
 \mu(B_\e)\frac12\iint |x-y|^2\gamma_\e+ \frac12\iint\limits_{B_\e ^c \times B_{R_0}}|x-y|^2\gamma.\\
\end{eqnarray*}
Since $\gamma_\e$ has marginals $\mu_\e, \mu_0$ then 
$\frac12\iint|x-y|^2\gamma_\e\ge d^2(\mu_\e, \mu_0)$. Consequently, this in combination with the last 
inequality yields
\begin{eqnarray}\nonumber
d^2(\mu, \mu_0)
&\ge&
 \mu(B_\e)d^2(\mu_\e, \mu_0)+\frac12 \iint\limits_{B_\e^c \times B_{R_0}}|x|^2\left(1-\frac{|y|}{|x|}\right)^2\gamma(x, y)dydx\\\nonumber
 &\ge&
 \mu(B_\e)d^2(\mu_\e, \mu_0)+\frac12 \iint\limits_{B_\e^c\times B_{R_0}}\left[|x|^2\left(1-\frac{R_0}{\e}\right)^2\gamma(x, y)dy\right]dx\\
 &=&
  \mu(B_\e)d^2(\mu_\e, \mu_0)+2c_0
  \int_{B_\e^c}|x|^2\mu\label{eq:d-est}, 
\end{eqnarray}
where we denote 
\begin{equation}\label{eq:c0-const}
c_0:=\frac14 \left(1-\frac{R_0}{\e}\right)^2
\end{equation} 
provided that $\e>{R_0}$.  
From H\"older's inequality we have  that 

\begin{eqnarray*}
2d^2(\mu, \mu_0)
&=&
\int|x|^2d\mu-2\iint x\cdot yd\gamma+\int|y|^2d\mu_0\\
&\ge&
 \frac12\int|x|^2d\mu-7 \int|y|^2d\mu_0, \\
\end{eqnarray*}
hence it gives
\begin{equation}\label{eq:wass-mu}
 \int|x|^2d\mu\le 4d^2(\mu, \mu_0)+14\int|y|^2d\mu_0\le 14(d^2(\mu, \mu_0)
 +R_0^2).
\end{equation}

{\bf Step 2:  A bound  for the logarithmic term:}

Now we want to estimate the 
logarithmic term from below. To do so we denote   
$Q(x)=c_0|x|^2, w(x)=e^{-c_0|x|^2}$ and introduce the logarithmic  energy with quadratic 
potential
\begin{eqnarray}\nonumber
I_w[\mu]
&=&
\iint\log \frac1{|x-y|}d\mu(x)d\mu(y)+2\int Qd\mu\\\label{blyaaa}
&=&
\iint\log \frac1{|x-y|w(x)w(y)}d\mu(x)d\mu(y).\\\nonumber
\end{eqnarray}
It is convenient to introduce the notation $K_w(x, y)=\log \frac1{|x-y|w(x)w(y)}$, with this we have
\begin{eqnarray*}
I_w[\mu]
&=&
\iint\limits_{B_\e\times B_\e}K_w(x, y)d\mu(x)d\mu(y)+2\iint\limits_{B_\e\times B_\e^c}K_w(x,y)d\mu(x)d\mu(y)\\
&&
+\iint\limits_{B_\e^c\times B_\e^c}K_w(x, y)d\mu(x)d\mu(y).
\end{eqnarray*}
 Observe that 
\begin{eqnarray*}
e^
{K_w(x, y)}=
\frac{e^{c_0(|x|^2+|y|^2)}}{|x-y|}
\ge
\frac{e^{c_0(|x|^2+|y|^2)}}{|x|+|y|} 
\ge 
\frac12\left(\frac{e^{2c_0(|x|^2+|y|^2)}}{|x|^2+|y|^2}\right)^{\frac12}
\end{eqnarray*}
because $\frac12(|x|+|y|)\le \sqrt{|x|^2+|y|^2}\le |x|+|y|$. Therefore for every large 
constant $T_0>0$ there is $\e$ such that if 
$\max\{|x|, |y|\}\ge \e$ then $K_w(x, y)\ge T_0$.  This yields the following estimate for 
$I_w$
\begin{eqnarray*}
I_w[\mu]
&\ge &
(\mu(B_\e))^2\iint\limits_{B_\e\times B_\e}K_w(x, y)d\mu_\e d\mu_\e+
2T_0\iint\limits_{B_\e\times B_\e^c}d\mu(x)d\mu(y)
+
T_0 \iint\limits_{B_\e^c\times B_\e^c}d\mu(x)d\mu(y)\\
&=&
(\mu(B_\e))^2\iint\limits_{B_\e\times B_\e}K_w(x, y)d \mu_\e d\mu_\e
+
2T_0\mu(B_\e)(1-\mu(B_\e))+T_0(1-\mu(B_\e))^2\\
&=&
(\mu(B_\e))^2\iint\limits_{B_\e\times B_\e}K_w(x, y)d\mu_\e d\mu_\e
+
T_0(1-(\mu(B_\e))^2). 
\end{eqnarray*}
Thus after some simplification we get

\begin{eqnarray}\nonumber
I_w[\mu]
&\ge&
(\mu(B_\e))^2 I_w(\mu_\e)+T_0(1-(\mu(B_\e))^2)\\\label{eq:log-est}
&\stackrel{\eqref{blyaaa}}{=}&
(\mu(B_\e))^2 \left[\iint \log \frac1{|x-y|}d\mu_\e d\mu_\e+ 2\int Q\mu_\e\right]+T_0(1-(\mu(B_\e))^2).\\\nonumber
\end{eqnarray}

{\bf Step 3: Energy comparison in $B_\e$:}

Combining \eqref{eq:log-est} with \eqref{eq:d-est} we get 

\begin{eqnarray*}
J[\mu]
&=&
\iint\log \frac1{|x-y|}d\mu(x)d\mu(y)+d^2(\mu, \mu_0)\\
&\stackrel{\eqref{eq:d-est}}{\ge}&
\iint\log \frac1{|x-y|}d\mu(x)d\mu(y)
+
\mu(B_\e)d^2(\mu_\e, \mu_0)+2c_0\int_{\R^n\setminus B_\e}|x|^2\mu\\
&=&
I_w(\mu)-2c_0\int_{B_\e}|x|^2d\mu + \mu(B_\e)d^2(\mu_\e, \mu_0)\\
&\stackrel{\eqref{eq:log-est}}{\ge} &
(\mu(B_\e))^2 \left[\iint \log \frac1{|x-y|}d\mu_\e d\mu_\e+ 2\int Q\mu_\e\right]+T_0(1-(\mu(B_\e))^2)\\
&& -2c_0\int_{B_\e}|x|^2d\mu + \mu(B_\e)d^2(\mu_\e, \mu_0)\\
&\ge&
(\mu(B_\e))^2J[\mu_\e] + 2c_0(\mu(B_\e))^2\int |x|^2\mu_\e+T_0(1-(\mu(B_\e))^2)  -2c_0\int_{B_\e}|x|^2d\mu. \\
\end{eqnarray*}
The last three terms on the last line can be further estimated from below as follows  
\begin{eqnarray}\nonumber
J[\mu]-(\mu(B_\e))^2J[\mu_\e] 
&=&
T_0(1-(\mu(B_\e))^2) + 2c_0\mu(B_\e)\int_{B_\e} |x|^2d\mu -2c_0\int_{B_\e}|x|^2d\mu\\\nonumber
&=&
T_0(1-(\mu(B_\e))^2) - 2c_0(1-\mu(B_\e))\int_{B_\e} |x|^2d\mu \\\nonumber
&=&
(1-\mu(B_\e))\left[T_0(1+\mu(B_\e))-2c_0\int_{B_\e} |x|^2d\mu\right]\\\label{jenya}
&\ge&
(1-\mu(B_\e))\left[T_0 -2c_0\int_{B_\e} |x|^2d\mu\right].\\\nonumber
\end{eqnarray}
In particular from here and \eqref{grbr} we see that $J[\mu]>-\infty$ and hence $(i)$ follows.
Now if we choose 
\begin{equation}\label{eq:T0-def}
T_0>1+J[\mu] +28c_0(d^2(\mu, \mu_0)+R_0^2)
\end{equation}
then from \eqref{jenya} it follows that 
\begin{eqnarray*}
J[\mu]-(\mu(B_\e))^2J[\mu_\e] 
&>& 
(J[\mu]+1)(1-(\mu(B_\e))^2) \\
&& + \ (1-\mu(B_\e))\left[28c_0(d^2(\mu, \mu_0)+R_0^2)(1+\mu(B_\e))- 2c_0\int_{B_\e} |x|^2d\mu\right] \\
&\stackrel{\eqref{eq:wass-mu}}{\ge}&
 (J[\mu]+1)(1-(\mu(B_\e))^2).
\end{eqnarray*}
This implies $(\mu(B_\e))^2(J[\mu]-J[\mu_\e])>1-(\mu(B_\e))^2$,  hence it is enough to take the minimization over $\mathcal M (B_{\e})$.

It remains to check (iii). First we estimate 
\begin{eqnarray*}
1+J[\mu_{k}] +28c_0(d^2(\mu_{k}, \mu_0)+R_0^2)
&\le&
1+J[\mu_{k}]+28c_0(d^2(\mu_{k}, \mu_0)+R_0^2)\\
&\le&
 1+C+28c_0(C+R_0^2)\\
 &\stackrel{\eqref{eq:c0-const}}{\le} &1+C+7(C+R_0^2):=\hat C.
\end{eqnarray*}
From \eqref{eq:T0-def} it follows that 
$T_0$ can be chosen to be the same for every $\mu_k$, say $T_0>\hat C$,  satisfying $0\le J[\mu_k]\le C$ and the proof is complete. 
\end{proof}

Now we are ready to finish the proof of  Theorem A.
\begin{thm}
Let $\rho_0\in \M 2$ such that $\supp \rho_0\subset B_{R_0}$ for some $R_0>0$. 
Then there exists a minimizer $\rho\in \M 2$ of $J$. Moreover, the support of $\rho$ is bounded. 
\end{thm}
\begin{proof}
First note that if we take the uniform measure $\mu$ of some ball $B$ having positive distance from $B_{R_0}$
then $J[\mu]<+\infty$. Hence by Proposition \ref{main-prop} (i) we have that  $J[\mu]>-\infty$. 
Thus if $\mu_k\in P_2(\R^2)$ is a minimizing sequence then without loss of generality we can assume that 
$J[\mu_k]\le C$ for some $C>0$ uniformly in $k$. 
Moreover, from Proposition \ref{main-prop} (ii) it follows that there are positive numbers $\e_k>0$ such that 
for the restriction measures $\mu_{k, \e_k}$ we have 
\begin{equation}\label{sandwich}
J[\mu_{k, \e_k}]<J[\mu_k]\le C.
\end{equation}
On the other hand it follows from \eqref{grbr} that $J[\mu_{k, \e_k}]\ge 0$ because $\supp \mu_{k, \e_k}$
is compact. Thus 
$0\le J[\mu_{k, \e_k}]\le C$ uniformly in $k$ and moreover $J[\mu_{k, \e_k}]\to \inf_{\rho\in P_2(\R^2)} J[\rho]$
thanks to \eqref{sandwich}. Consequently, applying Proposition \ref{main-prop} (iii),  we can use the weak compactness of $\mu_{k, \e_k}$ in 
$\mathcal M(B_{\e_0})$ to get a weakly converging subsequence still denoted  $\mu_{k, \e_k}$ to some 
$\rho\in \mathcal M(B_{\e_0})$.  
The logarithmic term is  lower-semicontinuous \cite{ST}, hence from  the lower-semicontinuity of $d$ (see property 4) in Section \ref{sec:2}) it follows that 
\[
J[\rho]\le \liminf_{k\to\infty}J[\mu_{k, \e_k}]
\]
and the desired result follows.
\end{proof}

\section{Random Matrices}\label{sec:4}
In this section we discuss a problem related to random matrices which 
leads to the obstacle problem \eqref{eq:main-obs}. 
Let  $H$ be a Hermitian matrix, i.e. 
$H_{ij}^\dagger=\bar H_{ji}$ (or $H^\dagger=H$ for short) where $\bar H_{ij}$ are the  complex conjugates of the entries 
of $N\times N$ matrix $H$. 
One of the well known random matrix ensembles is the Gaussian ensemble. The 
probability density of the random variables in the Gauss ensemble is given by the formula
\begin{equation*}
    P(H\in E)=\int_E e^{-\kappa {\rm Trace}( H^2)}dH,
\end{equation*}
where $\kappa>0$ and  
\begin{equation*}
    \hbox{Trace}  H^2=\sum_{ij}|H_{ij}|^2
\end{equation*}
is the trace of the squared matrix \cite{Mehta}. 
The dispersion is the same for every $H$ in the ensemble.

The corresponding statistical sum is
\begin{equation*}
    Z_{N}=\int e^{-\kappa{\rm Trace}( H^{2})}dH.
\end{equation*}
Regarding $H$ as a vector in $\mathbb C^{N^2}$ it is easy to see that 
the volume element is 
\begin{equation*}
    dH=\prod_{i=1}^NdH_{ii}\prod_{j<k}d\left(\Re H_{jk}\right)d\left(\Im H_{jk}\right).
\end{equation*}
Diagonalizing the matrix we have 
\[
H=UXU^\dagger, \quad X=\hbox{diag}(x_1, x_2, \dots, x_N), 
\]
where $U$ is a real unitary matrix $UU^\dagger =Id$, determined modulo a multiplication of 
$U_{\rm{diag}}=\mbox{diag}(e^{i\theta_1}, e^{i\theta_2}, \dots , e^{i\theta_N})$, thus we 
 consider the change of variables 
\[\left\{ H_{ik},H_{ii}\right\} \rightarrow \left\{ u_{ab},x_{a}\right\} \]
then 
\[dH=\left| J\right| \prod _{a\neq b}du_{ab}\prod _{a}dx_{a}, \]
where the Jacobian of the transformation is 
\[
J=\frac {\partial \left( H_{jk},H_{ii}\right) }{\partial \left( u_{ab},x_{a}\right) }, 
\]
which after some change of variables and simplifications leads to 
\[
|J|=A(u)\prod_{i<k}(x_i-x_k)^2.
\]
Since the trace of $H^2$ is invariant then it follows that 
$\mbox{Trace}( H^2)=\sum x_i^2$ and therefore
\[
\begin{split}
P(x_1, \dots, x_N)dx_1\dots dx_N
&=
\frac1{Z_N}\prod_{i<k}(x_i-x_k)^2\prod_{i}e^{-\kappa x_i^2}dx_i, \\
Z_N
&=
C_N\int
\prod_{i<k} (x_i-x_k)^2 \prod_{i} e^{-\kappa x_i^2} dx_i, 
\end{split}
\]
and $C_N$ is some universal constant (the volume of the unitary group factorized with respect to the subgroup of diagonal matrices).
The statistical sum $Z_N$ can be rewriten  in an equivalent form 
\[Z_N=C_N\int
\prod_{i<k} (x_i-x_k)^2 \prod_{i} e^{-\kappa x_i^2} dx_i=\int e^{-W}dx_1\dots dx_N, \]
where 
\[W=-\sum_{i\not =j}\log|x_i-x_j|+\frac Ng\sum_{i}x_i^2\]
and we replaced $\kappa=Ng$ for convenience. 
If we assume that the particles (in the equilibrium) have 
density $\rho$ then from approximation of Riemann's sum
we get that 
\[W\sim -N^2\int\int\log|x-y|\rho(x)\rho(y)dxdy+ {N^2}g\int\rho(x)|x|^2.\]
As $N\to \infty$ the main contribution comes from the minimum of the functional 
\[
F[\rho]=\int\int\log|x-y|\rho(x)\rho(y)dxdy+g\int\rho(x)|x|^2
\]with respect to 
the constraint $\int_\R \rho=1$.

If in $W$ the quadratic term is replaced by ${-\frac12|x_i-y_i|^2\gamma(x_i, y_i)}, H_0=\mbox{diag}(y_1, \dots, y_N)$,  then 
we get the model corresponding to the energy $J$. 
\begin{rem}
Let $n=1$, then  the first variation of  $F[\rho]$ gives
\[
-2\int_{\R}\log|x-y|\rho(y)+{x^2}g=\lambda,
\]
where $\lambda$ is the Lagrange multiplier of the constraint 
$\int_\R\rho=1$.
Differentiating in $x$ we get 
\[
\hbox{P.V.}\int_{-\infty}^{+\infty} \frac{\rho(y)dy}{x-y}=xg.
\]
The solution of this equation (given in terms of Hilbert's transform) has the form 
\[\rho(x)=\left\{
\begin{array}{ll}
\frac1{\pi g}\sqrt{\frac 2g-x^2} &\mbox{if}\ |x|<\sqrt{\frac2g},\\
0 &\mbox{if}\ |x|>\sqrt{\frac2g}, 
\end{array}
\right.
\]
and this is  Wigner's  famous semicircle law \cite{Serfaty}.

For the problem with $d^2$ we have $2U^\rho+\frac12|x-T(x)|^2=\lambda$, 
where $T:x\to y$ is the transport map. Since by Theorem B $x-T(x)=-2\frac{dU^\rho}{dx}$,  it follows
that $U^\rho+\left|\frac d{dx} U^\rho\right|^2=\lambda/2$.
Hence $U^\rho\le \lambda/2$ on $\supp\rho$ and 
\[
\pm \frac d{dx} U^\rho=\sqrt{\lambda/2-U^\rho}
\]
or equivalently 
$
\pm 2 \sqrt{\lambda/2-U^\rho}=x+C,
$
where $C$ is an arbitrary constant. Thus after normalization we get that 

\[
2U^\rho=\lambda-\frac{x^2}2 \quad \mbox{on}\  \supp \rho.
\]
\end{rem}


\section{Euler-Lagrange equation}\label{sec:6}

\begin{defn}
We say that a set $S\subset \R^n\times \R^n$ is cyclically monotone if 
\begin{equation}\label{eq:510}
\sum ^{m}_{k=1}\left| x_{k}-y_{k}\right| ^{2}\leq \sum ^{m}_{k={1}}\left| x_{k+1}-y_{k}\right| ^{2}
\end{equation}
holds whenever $m\ge 2$ and $(x_i, y_i)\in S, 1\le i\le m$ with $x_{m+1}=x_1$. 
The set $x_1, x_2, \dots, x_n$ is called a cycle.
\end{defn}
Cancelling the square terms from \eqref{eq:510} we get

\begin{equation}\label{urbs}
\sum ^{m}_{1}y_{k}x_{k}\geq \sum ^{m}_{1}y_{k}x_{k+1 }. 
\end{equation}
Let $\gamma$ be a transference plane with marginals $\rho, \rho_0$.
It is well known that the support of  $\gamma$
is cyclically monotone, see  \cite{Ambrosio} Theorem 2.2. 

Let $S\subset \mathbb{R} ^{n}\times \mathbb{R} ^{n}$ be cyclically monotone.
Set $c\left( x,y\right) =\dfrac {1}{2}\left| x-y\right| ^{2}$
and introduce the function 
\begin{equation}\label{2-star}
\psi \left( x\right) 
=\sup_{(x_{i},y_i)\in S}\left\{c\left( x_{0},y_{0}\right) -c\left( x_{1},y_{0}\right) +c\left( x_{1},y_{1}\right) -c\left( x_{2},y_{1}\right) +\ldots +c\left( x_{k},y_{k}\right) -c\left( x,y_{k}\right) \right\}, 
\end{equation}
where the supremum is taken over all cycles of finite length. 
It is easy to check that  $\psi$ defined in \eqref{2-star} satisfies 
$
\psi \left( x\right) \leq 0
$ 
and the normalization condition
$
\psi \left( x_{0}\right) =0.
$

If $\gamma(x, y)$ is a transference plan then it is contained in the 
$c$ superdifferential of the $c$ concave function $\psi$ constructed above. 
$\psi$ is called the maximal Kantorovich potential. 
Moreover, we have that 
if $(x', y')\in \supp\gamma$ then for every $x\in \R^n$  
\begin{equation}\label{albats}
\psi \left( x\right) +\dfrac {1}{2}\left| x-y'\right| ^{2}\geq \psi \left( x'\right) +\dfrac {1}{2}\left| x'-y'\right| ^{2}.
\end{equation}
See \cite{Ambrosio} for proof.

\begin{rem}
Recall that by Corollary 2.2 \cite{Ambrosio} if (CC) graphs are $\rho$ negligible then 
the transference plan $\gamma$ is unique and the transport map $T=\na v$ for some convex  potential $v$. 
\end{rem}


We want to show that in \eqref{albats}  we can take $\psi=2U^\rho$, and 
$\rho$ is absolutely continuous with respect to the Lebesgue measure. 

\begin{lem}\label{lem:sully}
$U^\rho \rho$ is a signed Radon measure.
\end{lem}

\begin{proof}
Let $\xi\in C_0^\infty(B)$
be a cut-off function of some ball $B$. Let $\{\rho_k\}_{k=1}^\infty$ be a sequence of 
mollifications of $\rho$. 
Recall that $I[\rho_k]<\infty$ and $\|\rho-\rho_k\|=I[\rho-\rho_k]\to 0$ as $k\to \infty$, see Remark \ref{rem:Fourier}. 
Thus 
\begin{eqnarray}
\int_B U^\rho\xi \rho_k
&=&
-\frac1{2\pi}\int_B U^\rho\xi \Delta U^{\rho_k}\\\nonumber
&=&
\frac1{2\pi}\int \na (U^\rho\xi)\na U^{\rho_k}\le \|\na(U^\rho\xi )\|_{L^2}^2\|\na U^{\rho_k}\|_{L^2}^2.
\end{eqnarray}
Note that \cite{Landkof} Lamma 1.$2'$ page 83
\begin{eqnarray*}
\|\na U^{\rho_k}\|_{L^2}^2
&=&
4\pi^2\int|\xi|^2|\widehat K|^2|\widehat \rho_k|^2 \\
&\stackrel{\eqref{eq:K-Fourier}}{\le} &
4\pi^2c_1\int \widehat K |\widehat \rho_k|^2
=
4\pi^2c_1 I[\rho_k]\\
&\le& 
8\pi^2c_1I[\rho]+8\pi^2c_1 I[\rho-\rho_k]
\to 
8\pi^2c_1 I[\rho]
\end{eqnarray*}
as $k\to \infty$. 
Since $U^\rho\in H^1$ (see \cite{K17}) is superharmonic (hence bounded below in $B$) then from Fatou's lemma we get that  
\[
 \int U^\rho\xi d\rho \le \lim _{k\to \infty}\int_B U^\rho\xi \rho_k\le C\|U^\rho\|_{H^1(B)}^2 I[\rho], 
\]
where $C$ depends only on the dimension. 
\end{proof}

\begin{thm}\label{thm:mon}
Suppose the infimum in $d(\rho, \rho_0)$ is realized for a transference plan 
$\gamma$ and $(x^\ast, y^\ast)\in \supp \gamma$. Then $\rho$ has $L^\infty$ density with respect to the Lebesgue
measure,  and for every $x_0$ we have 
\begin{equation}\label{eq:mon}
\frac12|x_0-y^\ast|^2-\frac12|x^\ast-y^\ast|^2+2U^\rho(x_0)-2U^\rho(x^\ast)\ge 0.
\end{equation}
Moreover, $\na U^\rho$ is log-Lipschitz continuous. 
\end{thm}

\smallskip
\begin{proof}
Let $\xi(x)$ be a cut-off function on $B_\e(x^\ast)$. 
Introduce 
\[
\gamma^\ast_\e(x, y)=\xi(x)\gamma(x, y)\with B_\e(x^\ast)\times B_\e(y^\ast).
\]
Note that $\gamma^\ast_\e(x, y)$ is not a probability measure.
Let $\gamma_\e(x, y)=\tau_\#\gamma^\ast_\e(x, y)$, 
where $\tau: (x, y)\to (x-x^\ast+x_0, y)$ is the translation operator in $x$ so that 
$\tau(x^\ast, y)=(x_0, y)$, see Figure \ref{fig1}. 
We see that the $x$ marginals are 

\begin{eqnarray}\label{eq:phi-star}
\varphi^\ast(x)&:=&\pi_{x\#}\gamma^\ast_\e(x,y),\\\nonumber
\varphi_0(x)&:=&\pi_{x\#}\gamma_\e(x, y).
\end{eqnarray}
Now we check the marginals in $y$

\[
\int _{}\left[ \gamma \left( x,y\right) -t\gamma ^{\ast }_{\varepsilon }\left( x,y\right) +t\gamma _{\varepsilon }\left( x,y\right) \right] dx=\int \gamma \left( x,y\right) dx=\rho_{0}\left( y\right) 
\]
because $T$ is measure preserving, and for the other marginal 
\[
\int
\left[ \gamma \left( x,y\right) -t\gamma ^{\ast }_{\varepsilon }\left( x,y\right) +t\gamma _{\varepsilon }\left( x,y\right) \right] dy=\rho\left( x\right) -t\varphi ^{\ast }(x)+t\varphi _{0}\left( x\right) .
\]
Observe that by \eqref{eq:phi-star} and the definition of $\gamma_\e^\ast$ we have 
\begin{eqnarray*}
\rho(x)-t\varphi^\ast(x)
&=&
\int _{}\left[ \gamma \left( x,y\right) -t\gamma ^{\ast }_{\varepsilon }\left( x,y\right)\right]dy\\
&=&
\int _{}\gamma(x, y)\left[1-t\xi(x)1_{B_\e(x^\ast)\times B_\e(y^\ast)}\right]dy\\
&\ge&
 0
\end{eqnarray*}
provided that $t$ is small enough. 

\medskip 
Consequently we can use $\rho-t\varphi^\ast+t\varphi_0$ against $\rho$ and get from the convexity of $d^2$  (see Section \ref{sec:2})
the following estimate 
\[
\begin{split}
d^{2}\left( \rho_{0},\rho-t\varphi ^{\ast }+t\varphi _{0}\right) 
&\leq 
\dfrac {1}{2}\iint \left| x-y\right| ^{2}d\left( \gamma {-t}\gamma ^{\ast }_{\varepsilon }+t\gamma _{\varepsilon }\right)\\ 
&=
d^{2}\left( \rho_{0},\rho\right) +\dfrac {t}{2}\iint \left| x-y\right| ^{2}d(\gamma_{\varepsilon }-\gamma ^{\ast }_{\varepsilon }).
\end{split}
\]
For the nonlocal term we have 
\[
\begin{split}
\iint K\left( x-y\right) d\left( \rho+t\left( \varphi _{0}-\varphi^\ast \right) \right) d\left( \rho+t\left( \varphi _{0}-\varphi^\ast \right) \right)
&=
\iint K(x-y)d\rho(x) d\rho(y)\\
&+2t\int U^{\rho }\left( x\right) d\left( \varphi _{0}-\varphi^\ast \right) +O\left( t^{2}\right). 
\end{split}
\]

\begin{figure}
\begin{center}

\begin{tikzpicture}
 \draw[gray, thick, -latex] (-1, 0) -- (6, 0) node[right]{$x$}; 
  \draw[gray, thick, -latex] (0, -1) -- (0, 3) node[left]{$y$}; 
  \draw[dashed] (0, 2)--(4.5, 2);
    \draw[dashed] (1.5, 0)--(1.5, 2);
    \draw[dashed] (4.5, 0)--(4.5, 2);

  \filldraw[black] (0,2) circle (2pt) node[left] {$y^\ast$}; 
  \filldraw[black] (1.5,0) circle (2pt) node[anchor=north] {$x^\ast$};
  \filldraw[black] (4.5,0) circle (2pt) node[anchor=north] {$x_0$};
  \filldraw[black] (1.5,2) circle (2pt);
  \filldraw[black] (4.5,2) circle (2pt);

\draw[pattern=north west lines, pattern color=red] (1,1.5) rectangle (2,2.5) node[above] {$\gamma^\ast_\e$};
\draw[pattern=north west lines, pattern color=blue!40] (4,1.5) rectangle (5,2.5) node[above] {$\gamma_\e=\tau_{\#}\gamma_\e^\ast$};

            \draw[-latex] (1.5, 2) to [bend left=-45] (4.5, 2) ;
            \node at (3, 1.1){$T$};

\end{tikzpicture}

\end{center}

\caption{The geometric construction of joint measures $\gamma_\e$ and $\gamma_\e^\ast$ via restriction and translation.  }
\label{fig1}
\end{figure}
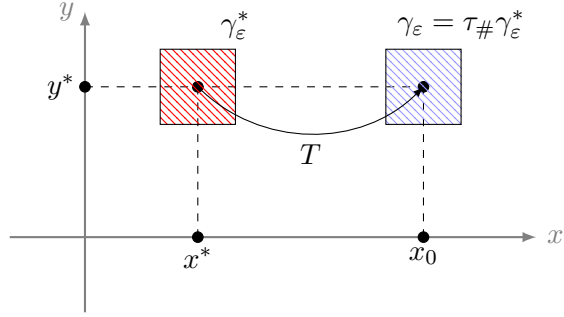

Then the energy comparison yields 
\[
\dfrac {t}{2}\iint \left| x-y\right| ^{2}d(\gamma_{\varepsilon }-\gamma ^{\ast }_{\varepsilon })
+2t\int U^{\rho }\left( x\right) d\left( \varphi _{0}-\varphi^\ast \right) +O\left( t^{2}\right)\ge 0.
\]
Sending $t\to 0$,  $t>0$ we get that 

\begin{equation}\label{eq:first-vat-glob}
\dfrac {1}{2}\iint  \left| x-y\right| ^{2}d(\gamma_{\varepsilon }-\gamma ^{\ast }_{\varepsilon })
+2\int U^{\rho }\left( x\right) d\left( \varphi _{0}-\varphi^\ast \right) \ge 0.
\end{equation}
Since $\gamma_\e$ is the push forward of $\gamma_\e^\ast$ under translation $x\to x-x^\ast+x_0$ then 
we have from \eqref{eq:first-vat-glob}
\begin{equation}\label{sully}
\dfrac {1}{2}\iint  \left[\left| x+x^\ast-x_0-y\right| ^{2}-|x-y|^2\right]d\gamma ^{\ast }_{\varepsilon }
+2\int \left[U^{\rho }\left( x+x^\ast-x_0\right)-U^\rho(x) \right]d\varphi^\ast \ge 0.
\end{equation}
Taking $x^\ast-x_0=\pm he_j$,  where $e_j$ is the unit direction of the $j$th coordinate axis, $h>0$,  and adding the resulted inequalities \eqref{sully}
we get 
\begin{eqnarray}\label{eq:sully3}
\dfrac {1}{2}\iint  \left[\left| x+he_j-y\right| ^{2}+|x-he_j-y|^2-2|x-y|^2\right]d\gamma ^{\ast }_{\varepsilon }\\\nonumber
+2\int \left[U^{\rho }\left( x+he_j\right)+U^{\rho }\left( x-he_j\right)-2U^\rho(x) \right]d\varphi^\ast \ge 0.
\end{eqnarray}
But $\left| x+he_j-y\right| ^{2}+|x-he_j-y|^2-2|x-y|^2=2h^2$,  hence \eqref{eq:sully3} is equivalent to
\begin{equation}\label{sully2}
-\int\delta_h U^\rho\xi d \rho
\le 
\frac12\int\xi d\rho.
\end{equation}
Note that by Lemma \ref{lem:sully} the left hand side of \eqref{sully2} is well defined.
 
\begin{claim}
$\rho$ has $L^2$ density.
\end{claim}
\begin{proof}
Let $\delta_hu=\delta(x, h, u)=\frac1{h^2}\sum_j(u(x+he_j)+u(x-he_j)-2u(x))$ be the
discrete Laplacian. Then from 
\eqref{sully2} with $\xi=1$ on $B_{\e_0}$  and recalling that $\rho$ has compact support,  it follows that 
\[
-\int \delta_h(U^\rho) d\rho\le \frac12\int d\rho=\frac12.
\]
Since $\supp\rho$ is compact we can assume that $K$ vanishes 
outside of $B_{r_0}$ and consider the truncated kernel 
$K_{r_0}=1_{B_{r_0}}K$. From the weak Parseval identity we get that 
\begin{eqnarray*}
\frac12
&\ge& 
-\int\widehat{ \delta_h U^\rho}\widehat{\overline \rho}
=
 -\int \frac1{h^2}\sum_j\left[e^{-2\pi i h\xi_j}+e^{2\pi ih\xi_j}-2\right]\widehat {U^\rho}\widehat{\overline \rho}\\
&=&
 -\int \frac1{h^2}\sum_j\left[e^{-2\pi i h\xi_j}+e^{2\pi i h\xi_j}-2\right]\widehat K_{r_0} |\widehat{ \rho}|^2\\
& =&
\frac1{h^2}\int \widehat K_{r_0} |\widehat{ \rho}|^2\sum_j2(1-\cos2\pi h\xi_j)
 =
 4\int \sum_j\frac{\sin^2(\pi\xi_j h)}{h^2}\widehat K_{r_0} |\widehat{ \rho}|^2.
\end{eqnarray*}
Letting $h\to 0$ and applying Fatou's lemma we get 
\begin{eqnarray*}
\frac12
&\ge&
  4\pi^2\int |\xi|^2\widehat K_{r_0} |\widehat{ \rho}|^2 \stackrel{\eqref{eq:K-Fourier}}{=} 4\pi^2c_1\int(1-\mathcal B(2\pi r_0|\xi|)) |\widehat{ \rho}|^2.
\end{eqnarray*}
Since the left hand side of the previous inequality does not depend on 
$r_0$ we can let $r_0\to \infty$ and applying Fatou's lemma again we see that 
\[
4\pi^2 c_1\int|\widehat{ \rho}|^2
\le
\frac12\int d\rho. 
\] 
Since Fourier transform is isometry on $L^2$ then $\tilde \rho$, the inverse Fourier transform of $\widehat \rho$,
exists and $\tilde \rho\in L^2$. But then  $\widehat{(\rho-\tilde \rho)}=0$,  and  it follows that 
$\rho $ has $L^2$ density. The proof of the claim is complete. 
\end{proof}

Returning to the localized inequality \eqref{sully2} with $(x^\ast, y^\ast)\in \supp\gamma$ we get
\[
-\int \delta_h(U^\rho) \xi\rho dx\le \int\xi\rho dx.
\] 
Using the weak convergence of second order finite differences in $L^2$ we finally obatin
\[
2\pi \int_{B_\e(x^\ast)}\rho^2\xi dx\le \int_{B_\e(x^\ast)}\rho\xi dx\le \left(\int_{B_\e(x^\ast)}\rho^2\xi dx\right)^{\frac12}\left(\int_{B_\e(x^\ast)}\xi dx\right)^{\frac12}.
\]
Consequently, the upper Lebesgue density  of the measure 
$\rho$ is bounded by some universal constant and hence $d\rho=fdx$ for some 
$f\in L^\infty(\R^n)$ \cite{EG}. 
Therefore from Judovi\v{c}'s theorem \cite{Yudovich } $\na U^\rho$ is log-Lipschitz continuous. 
Moreover, by construction 
\[
\int \varphi_0(x)=\int\varphi^\ast (x)=\iint \gamma_\e=\iint \gamma_\e^\ast.
\]
Hence from \eqref{eq:first-vat-glob} and the mean value theorem  we get that 
\[
\frac12|x_0-y^\ast|^2-\frac12|x^\ast-y^\ast|^2+2U^\rho(x_0)-2U^\rho(x^\ast)\ge 0.
\]
Thus 
$2U^\rho(x_0)+\frac12|x_0-y^\ast|^2\ge 2U^\rho(x^\ast)+ \frac12|x^\ast-y^\ast|^2$.
\end{proof}

\begin{cor}\label{cor:pizza}
Let $\rho$ be a minimizer of $J$, then $U^\rho=\psi$  on $\supp\rho$.
Furthermore, $\supp \rho$ has nonempty inetrior.
\end{cor}
\begin{proof}

In view of  \eqref{albats} and \eqref{eq:mon} $U^\rho$ and $\psi$ have the same c-subdifferential on  $\supp \rho$ then it follows that  
$U^\rho=\psi$ and at free boundary point $x^\ast=y^\ast$ we have 
$\na U^\rho(x^\ast)=0$. The last claim follows from 
the log-Lipschitz continuity of $\na U^\rho$.
\end{proof}

\section{The nonlocal Monge-Amp\`ere equation}\label{sec:7}
From Corollary \ref{cor:pizza} we have 
\[
y(x)=x+2\na U^\rho(x).
\]
Consequently,  the prescribed Jacobian equation is  
\[
\det(Id+2D^2 U^\rho)=\frac{\rho(x)}{\rho_0(x+2\na U^\rho)}.
\]
Note that this is a nonlocal Monge-Amp\`ere equation. By 
standard $W^{2, p}$ estimates for the potential $U^\rho$
it follows that $\supp\rho_0\setminus\supp\rho$ has vanishing Lebesgue measure.

Let $h>0$ be small and consider the perturbed energy
\[
\frac h2\iint K(x-y)d\rho d\rho+d^2(\rho, \rho_0).
\]
Linearizing  the equation 
\[
\det(Id+hD^2 U^\rho)=\frac{\rho(x)}{\rho_0(x+h\na U^\rho)}
\]
we 
\begin{eqnarray*}
\rho(x)
&=&
[1+h\Delta U^\rho+O(h^2)] \rho_0(x+h\na U^\rho)\\
&=&
[1+h\Delta U^\rho+O(h^2)] (\rho_0(x)+h\na \rho_0(x)\na U^\rho+O(h^2)).
\end{eqnarray*}
Consequently 
\[
\rho(x)-\rho_0(x)=h\na \rho_0(x)\na U(x)+h\Delta U^\rho(x) \rho_0(x)+O(h^2)
\]
or after iteration $\rho_0, \rho_1, \rho_2, \dots$ with step $\frac h2$ we get
\[
\rho_k(x)-\rho_{k-1}(x)=h\na \rho_{k-1}(x)\na U^\rho(x)+h\Delta U^\rho(x) \rho_{k-1}(x)+O(h^2). 
\]
Therefore, sending $h\to0$ we obtain the equation 
\[
\p_t \rho=\na \rho\na U^\rho+\Delta U^\rho \rho=\div(\rho\na U^\rho). 
\]

\section{Regularity of free boundary}\label{sec:8}
Let $x^\ast \in \supp \rho$, then  from \eqref{eq:mon} we have for every $x$ 
\[
U^\rho(x^\ast)\le U^\rho(x)+\frac14\left[|x-x^\ast|^2-|x^\ast-y^\ast|^2\right].
\]
Therefore $U^\rho(x^\ast)\le U^\rho(x)$ if $x\in B_{|x^\ast-y^\ast|}(x^\ast):=B$ and $x^\ast\not =y^\ast$.
Consequently  $U^\rho$ has local 
minimum in $\overline B$ at $x^\ast\in \p B$, and since $U^\rho$ is  superharmonic in $\R^2$ it
 follows from Hopf's lemma, applied to a ball with diameter $x^\ast y^\ast, $ that the normal derivative $\p_\nu U^\rho(x^\ast)<0$ 
 where $\nu=\frac{x^\ast-y^\ast}{|x^\ast-y^\ast|}$. 
Hence  at the remaining free boundary points we must have $x^\ast =y^\ast$ and hence $\na \psi(x^\ast)=0$.
 
\medskip

\begin{defn}
Let $T$ be the transport map. We say that $x\in \supp\rho\cap \supp \rho_0$ is a singular free boundary point if $x=T(x), \na U^\rho(x)=0$
and 
$$
\limsup_{t\downarrow 0}\frac1{|B_t|}\int_{B_t(x)}\rho=0.
$$ 
The set of singular points is denoted by $S$.
\end{defn}

\begin{lem}\label{lem:ellipse}
Let $0$ be a singular free boundary point and $\rho_0\ge s>0$ on $\supp\rho_0$. 
Then for every small $\e>0$ there is $R^\ast>0$ such that the set of singular points in $B_R, R<R^\ast$ can be trapped between two parallel planes 
at distance $\frac{\sqrt {8n+1}}{(sc_n)^{\frac1{2n}}} \e^{\frac1{2n}}R$
where $c_n=|B_1|$. 
\end{lem}
\begin{proof}
Let $\mathcal K$ be the convex hull of the singular set in $B_R$. Then there is 
$x_0\in B_R$ and an ellipsoid $E$ (John's ellipsoid \cite{deGuzman} page 139) so that 
\[
x_0+\frac1n E\subset \mathcal K\subset x_0+E.
\]

Let $r$ be the smallest axis of 
$E$. By mass balance condition 
\begin{equation}\label{eq:star1}
\int_{B_a} \rho(x)dx=\int_{T(B_a)}\rho_0(y)dy.
\end{equation}
By assumption $0$ is a singular point, so we have $\limsup\limits_{t\downarrow 0}\frac1{|B_t|}\int_{B_t} \rho(x)dx=0$.
Thus for every $\e>0$ small there is $a_0$ such that 
\begin{equation}\label{eq:star2}
\int_{B_a}\rho(x)dx\le \e a^n\quad  \mbox{whenever}\  a<a_0.
\end{equation}
By assumption $\rho_0>s>0$  then 
\begin{equation}\label{eq:star3}
\int_{T(B_a)} \rho_0
\ge 
\int _{B\frac {r}{2n}\left( x_{0}\right) }\rho _{0}\left( y\right) dy
 > 
 sr^{n}c_n
 \end{equation} while 
 $ \int _{B_a}\rho\left( x\right) dx < \varepsilon a^{n}. $ 
  
 Consequently, combining \eqref{eq:star1}-\eqref{eq:star3} we get 
$
\varepsilon a^{n} > s r^{n}c_n$
or  
\begin{equation}\label{eq:a}
a > \left[ \dfrac {s c_n }{\varepsilon }\right] ^{\frac {1}{n}}r.
\end{equation}

It follows that (for small $R$ and $\e$) there is 
a point 
$A\in B_{\frac {r}{2n}} \left( x_0\right)\cap \{\rho_0>0 \}$
and $B\in \{\rho>0\}$ so that $|OB| \sim a$ and  $T^{-1}\left( A\right) =B$.

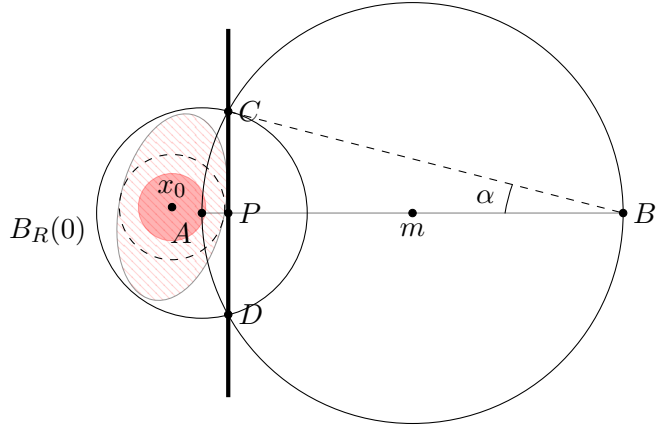
\begin{figure}
\begin{center}

\begin{tikzpicture}[scale=0.7]
\draw[black] (1, 0) circle (2cm);
\draw[pattern=north west lines, pattern color=red, rotate=78, opacity=0.4] (0.2,-0.4) ellipse (1.8cm and 1.0cm);
\filldraw[red, rotate=78, opacity=0.3] (0.2, -0.4) circle (0.64cm);
\draw[dashed, rotate=78] (0.2, -0.4) circle (1.0cm);

\draw[black] (5, 0) circle (4cm);
\draw[black, opacity=0.5] (1,0) -- (9, 0);
\draw[black, ultra thick] (1.5,3.5)  -- (1.5,-3.5) ;
\draw[black, dashed] (1.5,1.93) -- (9,0);

  \filldraw[black] (5,0) circle (2pt) node[anchor=north] {$m$};
  \filldraw[black] (9,0) circle (2pt) node[right] {$B$};
    \filldraw[black] (1,0) circle (2pt) node[anchor=north east] {$A$};
        \filldraw[black] (1.5,0) circle (2pt) node[anchor=west] {$P$};

  \filldraw[black] (1.5,1.93) circle (2pt)node[anchor=west] {$C$};
  \filldraw[black] (1.5,-1.93) circle (2pt)node[anchor=west] {$D$};
    \filldraw[black, rotate=78] (0.2,-0.4) circle (2pt) node[anchor=south] {$x_0$};
\draw (6.9,0.55) arc (160:171:3)node[anchor=south east] {$\alpha$};
    \filldraw[black]  (-1.0,-0.33) node[anchor=east] {$B_R(0)$};
\end{tikzpicture}
\end{center}
\caption{The construction used in the proof of Lemma \ref{lem:ellipse}. 
The red ball is $B_{\frac r{2n}}(x_0)$, and $T(A)=B$.}
\label{fig2}
\end{figure}


Let $x_s$ be a singular point. Notice that 
$x_{s}=T\left( x_{s}\right)$, i.e. the singular free boundary points are 
fixed points.
 From the monotonicity \eqref{urbs} 
\[
 \left( x_{s}-A\right) \left( T^{-1}\left( x_{s}\right) -T^{-1}\left( A\right) \right) \geq 0
 \]
or $\left( x_{s}-A\right) \left( x_{s}-B\right) \geq 0$. Let 
$m=\dfrac {A+B}{2}$ be the midpoint of the segment $AB$, then 
\begin{eqnarray*}
\left( x_{s}-A\right) \left( x_{s}-B\right) 
&=&
\left( x_{s}-B+B-A\right) \left( x_{s}-B\right) \\
&=&
\left| x_{s}-B\right| ^{2} +\left( B-A\right) \left( x_{s}-B\right)\\ 
&=&
\left| x_{s}-B\right| ^{2} -\left( A-B\right) \left( x_{s}-B\right)\\ 
&=&
\left| x_{s}-B\right|^2-2\dfrac {A-B}{2}\left( x_{s}-B\right) +\left| \dfrac {A-B}{2}\right| ^{2}-\left| \dfrac {A-B}{2}\right| ^{2}\\
&=&
|x_s-m|^2 -\left| \dfrac {A-B}{2}\right| ^{2}
\geq 0
\end{eqnarray*}
because 
\[
\left| x_{s}-B-\dfrac {A-B}{2}\right| ^{2}=\left| x_{s}-m\right| ^{2}.
\]
Hence  we arrive at 

\[
\left| x_{s}-m\right| ^{2}\geq \left| \dfrac {A-B}{2}\right| ^{2}.
\]
From simple geometric considerations we have that 
(see Figure 2)
$$\left| AP\right| =\left| AB\right| -\left| CB\right| \cos \alpha =\left| AB\right| -\left| AB\right| \cos ^{2}\alpha =\left| AB\right| \sin ^{2}\alpha.$$
Note that  
$\sin \alpha =\dfrac {\left| AC\right| }{\left| AB\right| }\le  \dfrac {2R}{\left| AB\right| }$, 
hence it follows that 

\[
\left| AP\right| \le \left| AB\right| \dfrac {4R^{2}}{\left| AB\right| ^{2}}=\dfrac {4R^{2}}{\left| AB\right| }.
\]

Therefore  $S\cap B_R$ is on one side of the hyperplane containing the intersection $B_R$ and the ball with diameter $AB$, see Figure \ref{fig2}.
Hence 
\[
\dfrac {r}{2n}\leq \dfrac {4R^{2}}{\left| AB\right| }
\]
or, in view of \eqref{eq:a},  we get $4R^{2}\geq \dfrac {r}{2n}\left[r\left( \dfrac {sc_n}{\varepsilon }\right) ^{1/n}-R\right]$.
From here 
$$
\dfrac {r^{2}}{2n}\left( \dfrac {sc_n}{\varepsilon }\right) ^{1/n}\leq R^{2}(4+\frac1{2n})
$$
implying $r \le  
\frac{\sqrt {8n+1}}{(sc_n)^{\frac1{2n}}} \e^{\frac1{2n}}R$ and the proof is complete. 
\end{proof}

\begin{lem}
Let $\om(R)$ be the height of the slab containing $S\cap B_R$ (see \eqref{small-o}),
$B_i=B_{r_i}(x_i)$ a collection of disjoint balls included in $B_R$ with $x_i\in S$. 
Then for every  $\beta>n-1$ we have 
\[
\sum r_i^\beta\le C\frac{R^\beta}{\om^{n-1}(R)}\frac1{1-\om^{\beta-(n-1)}(R)}
\]
\end{lem}

\begin{proof}
Rotate  the coordinate system 
such that $x_n$ points in the direction of the normal of the parallel planes
which are $\o{R}$ apart and contain $S\cap B_R$. 
Let $\F_0$ be the collection of the balls satisfying $R\om(R)<r_i\le R$. If $B_i\in \F_0$ then 
$\diam\(B_i\cap \{x_n=0\}\)\ge \frac12R\o R$. Therefore there are at most 
\[
\frac{R^{n-1}}{\(\frac12  R \o R\)^{n-1}}=\frac{2^{n-1}}{\(\o R\)^{n-1}}
\]
such balls. 
Thus we have 
$$\sum_{B_i\in \F_0}r^\beta_i\le \frac{2^{n-1}}{\(\o R\)^{n-1}}R^\beta$$
and $\{ B_i\}\setminus \F_0$ can be covered by balls $\widehat B_{4R\o R}(y_j)$ such that  
$y_j\in\{x_n=0\}\cap B_R$ and $1\le j\le \frac1{\(\o R\)^{n-1}}$. 
For each $j$ we have $S\cap \widehat  B_{4R\o R}(y_j)$ is contained in the slab of width 
$$
R\o R\(\o {R\o R}\)\le R\(\o R\)^2.
$$
 Hence  let  $\F_1$  be the collection of the balls $B_i$ contained in $\cup_j \widehat  B_{4R\o R}(y_j)$
 and satisfying $R\(\o R\)^2<r_i\le R\o R.$
 Then every ball $B_i$ in $\F_1$ intersects $\{x_n=0\}$ such that $\diam(B_i\cap \{x_n=0\})\ge \frac12 R\(\o R\)^2$
 and the number such balls $B_i$ is at most
 \[
 \frac{\(R\o R\)^{n-1}}{\(R\(\o R\)^2\)^{n-1}}=\frac1{\(\o R\)^{n-1}}.
 \]
 Consequently 
$$
\sum _{B_i\in \F_1}r_i^\beta=\frac1{\(\o R\)^{n-1}}\sum_{B_i\in \widehat B_{R\o R}(y_1)}\(R\o R\)^\beta\le \frac{R^\beta}{\(\o R\)^{2(n-1)}}.
$$
 
 Again, as above  we can choose at most $\frac1{\(\o R\)^{n-1}}$ balls $\widehat B_{R\(\o R\)^2}(y_l), l\le \frac1{\(\o R\)^{n-1}}$
 that cover $\{B_i\}\setminus(\mathcal F_0\cup\mathcal F_2)$.
 We define $\F_m$ inductively such that  
 $R\(\o R\)^{m}<r_i\le R\(\o R\)^{m-1}$ for $B_i\in \F_m$, then repeating the argument above we have that 
 $$\sum_{B_i\in \F_m}r_i^\beta\le \(\frac1{\(\o R\)^{n-1}}\)^{m+1}\(R\(\o R\)^m\)^{\beta}.$$
 Therefore 
 \begin{eqnarray*}
 \sum_i r_i^\beta
 &\le&
\sum_{m=0}^\infty \sum_{B_i\in \F_m} r_i^\beta\le \sum_{m=0}^\infty \(\frac1{\(\o R\)^{n-1}}\)^{m+1}\(R\(\o R\)^m\)^{\beta}=\\
&=&
\frac{R^{\beta}}{\(\o R\)^{n-1}}\sum _{m=0}^\infty \(\(\o R\)^{\beta-(n-1)}\)^m\\\
&=&
\frac{R^{\beta}}{\(\o R\)^{n-1}}\frac1{1-\(\o R\)^{\beta-(n-1)}}.
 \end{eqnarray*}

\end{proof}
Now we can finish the proof of Theorem C. 
\begin{thm}
Suppose $\o R= R^\sigma$, then there is $\sigma'>0$ depending only on 
$n, \sigma$ such that $S\subset M_0\cup\bigcup_{i=1}^\infty M_i$
where $\H^{n-1-\sigma'}(M_0)=0$ and $M_i$ is a  $C^1$ hypersurface 
such that the measure theoretic normal exists at each $x\in S\cap M_i, i\ge 1$.
\end{thm}

\begin{proof}
Let $x\in S$ be such that there exists a unique normal in measure theoretic sense, see Definition 5.6 \cite{EG}.
Notice that at the point $x$,  where such normal exists  the set has approximate tangent plane.
Therefore  the projections of $B_r(x)\cap S$ onto two dimensional planes have diameter 
at least $2R$. Thus we let $M_0$ be the subset of $S$ such that 
for $x\in M_0$ there is sequence $R_k\to 0$ such that the projections of 
$B_{R_i}(x)$ onto some two dimensional plane is of order $R^{1+\sigma}$. 

Now let $B_{r_i}(x_i)$ be a Besikovitch type covering of $B_R\cap M_0$.
Let us cover $B_{r_i}(x_i)\cap M_0$ with  balls  of radius $r_i^{1+\frac\sigma2}$, then there are at most 
\[
\frac{r_i^{n-2}}{r^{(n-2)(1+\frac\sigma2)}_i}=\frac1{r^{\frac\sigma2(n-2)}_i}
\]  
such balls. 
Hence for $\alpha>0$ we have 
\[
\sum_i r_i^\alpha\le \sum_i  \frac1{r^{\frac\sigma2(n-2)}_i}r_i^{\alpha(1+\frac\sigma2)}=\sum_i r_i^{\alpha(1+\frac\sigma2)-\frac\sigma2(n-2)}.
\]
Now we choose $\delta=\frac\sigma4$ and $\beta:=n-1+\delta$
and set 
\begin{eqnarray*}
\beta
&:=&
\alpha(1+\frac\sigma2)-\frac\sigma2(n-2)=n-1+\delta.
\end{eqnarray*}
We want to show that for this choice of $\beta$ we get $\alpha=n-1-\sigma'$
for some $\sigma'>0$ depending on $n$ and $\sigma$. Indeed, we have  
\begin{eqnarray*}
\alpha
&:=&
\frac{(n-1)+\delta+\frac\sigma2(n-2)}{1+\frac\sigma2}=
\frac{(n-1)+\frac\sigma4+\frac\sigma2(n-2)}{1+\frac\sigma2}\\
&=&
\((n-1)+\frac\sigma4+\frac\sigma2(n-2)\)\(1-\frac\sigma 2+o(\sigma)\)\\
&=&
n-1+\frac\sigma4(1+2(n-2)-2(n-1))+o(\sigma)\\
&=&
n-1-\frac\sigma 4+o(\sigma)\ge n-1-\sigma'.
\end{eqnarray*}
\end{proof}

\section*{Acknowledgments} 
The author wishes to express his thanks  to Prof. Robert McCann for several helpful comments concerning the related literature.

\end{document}